\documentclass[11 pt]{amsart}
\usepackage[applemac]{inputenc}

\usepackage[left=30mm,right=30mm,top=30mm,bottom=30mm]{geometry}
\usepackage[dvipsnames,svgnames,x11names,hyperref]{xcolor}
\usepackage[hyphens]{url}
\usepackage[colorlinks=true,linkcolor=Maroon,citecolor=blue,urlcolor=blue,hypertexnames=false,linktocpage]{hyperref}
\usepackage{bookmark}
\usepackage{amsmath,thmtools,mathtools}
\mathtoolsset{showonlyrefs=true}

\usepackage{blindtext}
\newcounter{mparcnt}

\usepackage{fancyhdr}
\usepackage{esint}
\usepackage{enumerate}

\usepackage{pictexwd,dcpic}
\usepackage{graphicx}
\usepackage{caption}
\usepackage[dvipsnames]{xcolor}
\swapnumbers
\declaretheorem[name=Theorem,numberwithin=section]{thm}
\declaretheorem[name=Remark,numberwithin=section,style=remark,sibling=thm]{rem}
\declaretheorem[name=Lemma,sibling=thm]{lemma}
\declaretheorem[name=Proposition,numberwithin=section,sibling=thm]{prop}
\declaretheorem[name=Definition,style=definition,numberwithin=section,sibling=thm]{defn}
\declaretheorem[name=Corollary,numberwithin=section,sibling=thm]{cor}

\numberwithin{equation}{section}

\newcommand{\wt}{\widetilde}
\newcommand{\wh}{\widehat}

\newcommand{\cn}{\colon}
\newcommand{\sub}{\subset}
\newcommand{\ov}{\overline}

\newcommand{\bbR}{\mathbb{R}}

\newcommand{\al}{\alpha}
\newcommand{\be}{\beta}
\newcommand{\ga}{\gamma}

\newcommand{\ep}{\epsilon}

\newcommand{\si}{\sigma}
\newcommand{\Si}{\Sigma}

\newcommand{\Om}{\Omega}

\newcommand{\cC}{\mathcal{C}}

\newcommand{\del}{\partial}
\newcommand{\n}{\nabla}

\newcommand{\rt}{\sqrt}
\newcommand{\ho}{\left(\frac{d}{dt} -\Delta \right)}

\newcommand{\ip}[2]{\left\langle #1,#2 \right\rangle}
\newcommand{\fr}[2]{\frac{#1}{#2}}

\newcommand{\pf}[1]{\begin{proof}#1 \end{proof}}
\newcommand{\eq}[1]{\begin{equation}\begin{alignedat}{2} #1 \end{alignedat}\end{equation}}

\newcommand{\br}[1]{\left(#1\right)}
\newcommand{\abs}[1]{\lvert #1\rvert}
\newcommand{\enum}[1]{\begin{enumerate}[(i)] #1 \end{enumerate}}

\newcommand{\ra}{\rightarrow}

\newcommand{\mrm}{\mathrm}

\begin{document}
\title[Spacelike MCF with constant angle in a cone]{A capillary problem for spacelike mean curvature flow in a cone of Minkowski space}
\author{Wilhelm Klingenberg}
\address{Department of Mathematical Sciences, Durham University Upper Mountjoy Campus, Stockton Rd, Durham DH1 3LE, UK}
\email{\href{mailto:wilhelm.klingenberg@durham.ac.uk}{wilhelm.klingenberg@durham.ac.uk}}
\author{Ben Lambert}
\address{School of Mathematics, The University of Leeds, Leeds, LS2 9JT, UK}
\email{\href{mailto:b.s.lambert@leeds.ac.uk}{b.s.lambert@leeds.ac.uk}}
\author{Julian Scheuer}
\address{Goethe-Universit\"at Frankfurt am Main, Institut f\"ur Mathematik, Robert-Mayer-Str.~10, 60325 Frankfurt, Germany} 
\email{\href{mailto:scheuer@math.uni-frankfurt.de}{scheuer@math.uni-frankfurt.de}}
\date{\today}

\begin{abstract}
 Consider a convex cone in three-dimensional Minkowski space which either contains the lightcone or is contained in it. This work considers mean curvature flow of a proper spacelike strictly mean convex disc in the cone which is graphical with respect to its rays. Its boundary is required to have constant intersection angle with the boundary  of the cone. We prove that the corresponding parabolic boundary value problem for the graph admits a solution for all time which rescales to a self-similarly expanding solution.
\end{abstract}

\maketitle

\section{Introduction}

We study the capillary problem for mean curvature flow of spacelike surfaces $M_t$ with free boundary on a non-degenerate (i.e. Riemannian or Lorentzian) surface $\Si$ in Minkowski space. It is well known that spacelike mean curvature flow with gradient estimates is well behaved, see for example \cite{Bartnik:/1984,Ecker:/1993,Ecker:/1997,EckerHuisken:/1991,LambertLotay:02/2021}. One particularly intriguing class of  capillary boundary conditions is therefore given by $\Si$ being a Riemannian submanifold, as this immediately implies uniform spacelikeness at the boundary (depending only on barriers), which in turn yields global gradient estimates in the compact case. However, as we shall see, this necessarily leads to other issues such as the possibility of boundary collapse. In this paper we deal with the capillary problem when $\Si$ is the boundary of any convex cone with a nondegenerate induced metric. This problem carries with it several technical difficulties: Typically a surface being spacelike implies that the surface is graphical -- this is not true in the case $\Si$ is a Riemannian cone, but a property that we must show is preserved along the flow. A second difficulty is the appearance of unwanted curvature terms in boundary derivatives of first order, terms which are present due to the non-perpendicular boundary condition. We deal with this by exploiting two-dimensional techniques and good bounds on the mean curvature.

Given any such cone $\Si$, mean curvature flow with a capillary boundary condition is given by the following PDE 
\eq{
\label{eq:MCFBC}
\begin{cases}
(\del_{t}x)^\perp=\vec{H} = H\nu & \text{on }\Om\times[0,T)\\
x(\cdot,0) = x_0(\cdot)&\text{in } \Om\\
x(\xi,t)\in \Sigma &\text{for } (\xi,t)\in \partial \Om\times[0,T)\\
-\ip{\nu(\xi,t)}{\mu(x(\xi,t))}=\alpha & \text{for } (\xi,t)\in \partial \Om\times[0,T),
\end{cases}
}
for some fixed $\al\in\bbR$.  Furthermore, $\mu$ is the future-directed normal to $\Si$, see \eqref{mu}.
Let $\cC\in \bbR^3_1$ be an open convex cone such that $\Si = \del\cC$. Our convex cone $\cC$ arises from the following construction.
Let $\Om\sub \bbR^2$ be a convex, bounded and open domain containing the origin. We may also view $\Om$, without renaming it, as a subset of $\bbR^{3}_1$ within the slice $\{x^{3}=1\}$. Then $\Om$ generates an open convex cone $\cC$ in $\bbR^{3}_1$ by radial extension with apex being the origin of $\bbR^{3}_1$, i.e. $\Om$ and $\cC$ are related via
\eq{\Om = \cC\cap \{x\in \bbR^{3}_{1}\cn \ip{x}{e_{3}}=-1\}.}

The boundary condition in \eqref{eq:MCFBC} has a special name.

\begin{defn}
For $\al\in \bbR$ when $\Si$ is Lorentzian and $\al>1$ when $\Si$ is Riemannian, we say that a surface $M\sub \cC$ with boundary $\del M$ is $\al$-{\it{capillary}}, if $\del M\sub\Si$ and along $\del M$ there holds
\eq{-\ip{\nu}{\mu} = \al.}
\end{defn} 

Finally, we say that $M$ is {\it graphical}, if we can parametrise $M$ via its graph function 
\eq{u\cn\Omega\ra \bbR,}
so that the embedding $x$ of $\Om$ into $\cC$ is given by
\eq{x(\xi)=u(\xi)(\xi+e_{3}).\label{eq:graphicalx}}
Here we identify $\xi \in \Om  \subset  \bbR^2 \times \{1\} \sim \bbR^2 \times \{0\} \subset \bbR^{3}_1,$ so that $\ip{\xi}{e_3}=0$.

Our main result is the following.
\begin{thm}\label{thm:maintheorem}
Suppose that $\cC\subset{\mathbb{R}^3_1}$ is a strictly convex cone with nondegenerate boundary $\Si$ and suppose that $M_0\subset \mathcal{C}$ is $\al$-capillary, spacelike, graphical and strictly mean-convex. 
Then: 
\begin{enumerate}
\item A solution of \eqref{eq:MCFBC} exists for all time and leaves any compact subset of ${\mathbb{R}^3_1}$ for large $t$. 
\item We define the rescaled flow to be $\wt{M}_t:=(1+t)^{-\frac 1 2}M_t$. The rescaled solution stays in a bounded region and converges uniformly to a piece of an expanding solution to MCF satisfying the boundary conditions as $t \ra \infty$.
\end{enumerate}
\end{thm}
\begin{rem}
\enum{
\item We require $\al\geq 1$ in the Riemannian case, since by properties of future directed timelike vectors $\alpha=-\ip{\nu}{\mu}\geq 1$. In fact, for this to be an oblique boundary condition for the PDE we require the strict inequality, see \autoref{lem:obliquerho}.
\item Note that in the above not every $\alpha$ admits mean convex data in the Lorentzian case. However, it can easily be seen that for $\alpha<0$ such initial data always exists.
\item For a full treatment of regularity at the initial time see \autoref{thm:maintheoremPDE}.
}

\end{rem}

Mean curvature flow in semi-Riemannian manifolds has seen a great deal of interest over the years, with applications in producing prescribed mean curvature hypersurfaces in the works of K. Ecker and G Huisken \cite{EckerHuisken:/1991} and C. Gerhardt \cite{Gerhardt:09/2000} and in producing homotopies between manifolds by G. Li and I. Salavessa \cite{LiSalavessa:/2011}. A variant of mean curvature flow constrained to lie inside null hypersurfaces has recently been used by H. Roesch and the third author as a method to find MOTS in general relativity \cite{RoeschScheuer:02/2022}, with an interesting connection to Yamabe-flow \cite{Wolff2023}.

Mean curvature flow with boundary conditions in semi-Riemannian spaces has been considered by a range of authors such as K. Ecker \cite{Ecker:/1997}, the second author \cite{LambertOblique2012,Lambert2017} and S. Gao, G. Li and C. Wu \cite{GaoLiWu2014}. In particular, the case of a perpendicular boundary condition on a Lorentzian cone was proven in all dimensions by the second author \cite{Lambert2014} (equivalent to $\al=0$ in the present paper). Similar perpendicular timelike boundary conditions have also been considered for a mean curvature type flow by F. Guo, G. Li and C. Wu \cite{GuoLiChuanxi2014}. In the Riemannian setting a wide range of boundary problems have been considered for mean curvature flow and we do not give a full bibliography here but mention the works of G. Huisken \cite{Huisken:02/1989}, A. Stahl \cite{Stahl1996}, A. Freire \cite{Freire2014}, B. Guan \cite{Guan1996}, V. Wheeler \cite{Wheeler2017},  and N. Edelen, R. Haslhofer, M. Ivaki and J. Zhu \cite{EdelenHaslhoferIvakiZhu2022}.  


As mentioned in the introduction, one might hope for good behaviour from spacelike mean curvature flow with uniform spacelike estimates (e.g. by applying interior estimates such as in \cite{Ecker:/1997, LambertLotay:02/2021}). However, we note that in the case that $\Si$ is Riemannian, even with uniform spacelike estimates, boundary singularities can occur via boundary collapse. Indeed, the self-shrinker examples of Halldorsson \cite[Theorem 7.1]{Halldorsson2015} provide self shrinkers satisfying capillary boundary conditions with $\Sigma$ the $x$-axis in $\mathbb{R}^2_{1}$ (for all angles) and form singularities, with the whole boundary shrinking to a point. We expect similar rotationally symmetric self shrinkers to exist in $\mathbb{R}^3_1$ indicating the necessity of the strict convexity assumption on $\Sigma$.

The method of proof is to prove uniform graphicality and spacelikeness estimates, which ultimately allow us to rewrite  \eqref{eq:MCFBC} as a uniformly parabolic PDE with an oblique nonlinear boundary condition, see \eqref{eq:MCFrescaledrho}. The proof is different in the cases that $\Si$ is either Riemannian or Lorentzian. In fact, the case that $\Si$ is Riemannian requires significantly more work to prove graphicality estimates, while uniform spacelikeness follows more or less for free as mentioned above. The major difficulty in both cases is to get boundary estimates and these are obtained using estimates on the mean curvature and careful use of the maximality condition as in \cite{AltschulerWu1994,Lambert2019}.

In section \ref{sec:GeometricQuantities} we describe this reduction and give sufficient conditions for obliqueness and parabolicity, in  \autoref{prop:parabolicityobliqueness}. In section \ref{sec:BoundaryIdentities} we collect the required boundary identities from an analysis of our boundary condition. In section \ref{sec:InitialEstimates} we prove initial estimates which are true in both the case that $\Sigma$ is Lorentzian or Riemannian, including height estimates and a speed limit. Our proof then splits into two cases:
\begin{itemize}
\item If $\Sigma$ is Riemannian: A uniform spacelikeness bound follows immediately in \autoref{lem:presspacelike}, so the key estimate required is a uniform graphicality estimate. This is completed in section \ref{sec:SpacelikeCone}.
\item If $\Sigma$ is Lorentzian: The flow is automatically graphical, and we require a uniform spacelikeness estimate. This is proven in section \ref{sec:LorentzCone}.
\end{itemize}  
Finally, we prove the full theorem in section \ref{sec:ProofOfTheorem}, see \autoref{thm:maintheoremPDE} for full details.

\subsection*{Acknowledgments}
The authors acknowledge Cardiff, Durham and Leeds Universities for hosting visits during which the present work was carried out.

\section{Geometric quantities in cone coordinates}
\label{sec:GeometricQuantities}

\subsection*{Spacelikeness and graphicality}

In this section we investigate the property of being graphical further and deduce formulae for geometric quantities in terms of the graph parametrisation as in \eqref{eq:graphicalx}. 
First of all it is important to note that spacelikeness does not in general imply graphicality if $\Si$ is Riemannian, in contrast to the more common case that the surface is graphical over the flat Euclidean subspace. Later we will show that under mean curvature flow, each of these properties is preserved in a quantitative sense.

For the sake of compressing some formulas it will occasionally be useful to work with 
\eq{\rho = \log u}
instead of $u$.

With the standard basis $(e_{1},e_{2})$ in $\Om$, the
 tangent vectors are
\eq{x_i = ue_i+D_iu(\xi+e_{3}),}
where $D$ is the standard directional derivative in $\Om$,
and in these coordinates the metric induced on $M$ is
\eq{g_{ij} &= u^2 \delta_{ij} +u(\xi_iD_ju+\xi_jD_iu)+D_iuD_ju(|\xi|^2-1)\\
&= e^{2\rho} \left[\delta_{ij} +(\xi_iD_j\rho+\xi_jD_i\rho)+D_i\rho D_j\rho(|\xi|^2-1)\right].
}
Its inverse is given by
\eq{g^{jk} = e^{-2\rho}\left(\delta^{jk} + \frac{D^j\rho D^k\rho +|D\rho|^2\xi^j\xi^k-(1+D\rho\cdot\xi)(\xi^jD^k\rho+\xi^kD^j\rho)}{(1+D\rho\cdot \xi)^2-|D\rho|^2}\right).\label{eq:ginverse}}

By inspection we see that
\eq{\widetilde{\nu}:=Du+(u+\xi\cdot Du)e_{3}}
satisfies 
\eq{\ip{\widetilde{\nu}}{x_i}= D_iu(u+Du\cdot \xi)-D_iu(u+Du\cdot \xi)=0.}
 As at $\xi=0$, $\wt{\nu}$ has nontrivial positive $e_{3}$ part, we see that this is a future directed normal which must be timelike due to the assumption on $M$. Hence we obtain
 \eq{0>\abs{\wt\nu}^{2} = \abs{Du}^{2}-(u+\xi\cdot Du)^{2}}
 and hence 
 \eq{(u+\xi\cdot Du)^2>|Du|^2.}
  This gives us a globally defined unit normal
\eq{\label{nu}\nu = \frac{Du+(u+\xi\cdot Du)e_{3}}{\sqrt{(u+\xi\cdot Du)^2-|Du|^2}}.}
Therefore, the standard measure of spacelikeness satisfies
\eq{1\leq v:=-\ip{\nu}{e_{3}} = \frac{(u+\xi\cdot Du)}{\sqrt{(u+\xi\cdot Du)^2-|Du|^2}}=\frac{1}{\sqrt{1-\frac{|Du|^2}{(u+\xi\cdot Du)^2}}}.}
The support function of a surface $M \subset \bbR^{3}_1 $ is given by 
\eq{S:=-\ip{x}{\nu} = - \frac{u Du\cdot\xi-uDu\cdot\xi-u^2}{\sqrt{(u+\xi\cdot Du)^2-|Du|^2}}=\frac{u^{2}}{\sqrt{(u+\xi\cdot Du)^2-|Du|^2}}.}

The Gaussian formula on $M$ is given by
\eq{\overline{\n}_{X}Y = \n_{X}Y + \mrm{II}(X,Y),}
where $\overline{\nabla}$ is the Levi-Civita connection on $\mathbb{R}^3_1$, and so the second fundamental is given by
\eq{h_{ij} = -\ip{\mrm{II}(\del_{i},\del_{j})}{\nu} = -\ip{D_{ij}x}{\nu}.}
With this convention the mean curvature of $M$ is given by
\eq{H &=-\ip{2g^{ij}D_{j}u e_{i} + g^{ij}D_{ij}u(\xi+e_{3})}{\nu}\\
	&=-\fr{2g^{ij}D_{i}uD_{j}u}{\rt{(u+\xi\cdot Du)^{2}-\abs{Du}^{2}}}+\fr{ug^{ij}D_{ij}u}{\rt{(u+\xi\cdot Du)^{2}-\abs{Du}^{2}}}.}

Geometric quantities on $\Si$ will be furnished by a hat, e.g. the second fundamental form is defined by
\eq{\widehat h_{ij}=\ip{\overline\n_{i}\mu}{y_{j}},}
where $y$ is an embedding of $\Si$, and where the future-directed normal $\mu$ is defined by
\eq{\label{mu}\mu = \fr{N + N\cdot\xi \;\; e_{3}}{\rt{\abs{1-(N\cdot\xi)^{2}}}}.}
Here $N$ is the outward pointing unit normal to $\del\Om\sub\bbR^2$ (and identified with a vector in $\bbR^2 \times \{0\}$),  i.e. $N\cdot\xi$ is the standard Euclidean support function on $\bbR^2$. The normal vector $\mu$ is well defined due to the metric non-degeneracy of $\Si$. To keep track of the signature of $\Si$, we also define
\eq{\si = \ip{\mu}{\mu} = \begin{cases} +1, &\Si~\mbox{is Lorentzian}\\
								-1, &\Si~\mbox{is Riemannian}.
				\end{cases}}
Note also that the sign of $\sigma$ coincides with the sign of $1-(N\cdot\xi)^{2}.$ We next deduce some geometric properties of $\al$-capillary surfaces.

For any $p\in \mathbb{R}_1^3$ and $W\in T_p\mathbb{R}_1^3$, when dealing with semi-Riemannian norms of vectors of unknown causality we write
\[\|W\|^2=||W|^2| = \abs{\ip{W}{W}}.\]

For $p\in \partial M_t$ and a vector $W\in T_p\mathbb{R}^{3}_1$, we will write 
\eq{W^\top=W+\ip{\nu}{W}\nu} for the orthogonal projection of $W$ to $T_p M_t$ and 
\eq{W^\Sigma=W-\sigma\ip{\mu}{W}\mu} for the orthogonal projection of $W$ to $T_p\Sigma$. We will write $W^{M\cap\Si}$ for the projection of $W$ to $T_p \partial M = T_p M_t\cap T_p\Sigma$. 
In particular we have
\begin{align*}
\mu^\top &= \mu - \alpha \nu, & \nu^\Sigma &= \nu +\sigma \alpha \mu,\\
|\mu^\top|^2 &=\al^2+\sigma, &|\nu^\Sigma|^2 &= -\sigma(\al^2+\sigma)=-\si\|\nu^\Si\|^2.
\end{align*}
We note that $\mu^\top$ and $\nu^\Sigma$ are orthogonal to $\partial M_t$ and contained in $TM$ and $T\Sigma$ respectively.

\begin{lemma}
Let $M\sub \cC$ be $\al$-capillary. Then $\mu^\top$ is normal to $\partial M \subset \Si = \partial \cC $  and points out of $\cC$. 
\end{lemma}
\begin{proof}
For $p\in \del M$, $T_p \partial M$ is orthogonal to both $\nu$ and $\mu$. We immediately see that $\mu^\top(p)$ is orthogonal to $T_p \partial M$. In the case that $\Si$ is Lorentzian, a vector $V$ points out of $\mathcal{C}$ at $p\in\partial M$ if it is on the same side of $T_p\Sigma$ to $\mu$. If $\Si$ is Riemannian, then $V$ points out of $\mathcal{C}$ if it is on the opposite side of $T_p\Sigma$ to $\mu$. In both cases, $\ip{V}{\mu}>0$. The Lemma follows by calculating
\eq{\ip{\mu}{\mu^{\top}} = \ip{\mu}{\mu-\alpha\nu} =\al^{2}+\si>0.}
\end{proof}
\subsection*{Second fundamental form of \texorpdfstring{$\Si$}{Sigma}}
We suppose that $z\cn [0,\ell(\partial \Omega)] \ra \mathbb{R}^2$ is an arc-length para\-metrisation of $\partial \Omega$, where $\ell(\partial \Omega)$ is the length of $\partial \Omega$. We will write differentiation of $z$ with respect to arc-length parameter, $s$, by $\dot{z}$. We may now parametrise $\Si$ by $\wh{z}\cn [0,\ell(\partial\Om)]\times(0, \infty)\ra \mathbb{R}^{3}_1$ given by  
\[\wh{z}(s,\lambda) = \lambda (z(s)+e_{3}) .\]
It is easy to see that $\wh{h}$ has a zero eigenvector in the $\frac{\partial}{\partial\lambda}$ direction while in the $\partial \Omega$ directions, recalling \eqref{mu},
\[\wh{h}\left(\dot{z},\dot{z}\right) = -\lambda^{-2}\ip{\wh{z}_{ss}}{ \mu} 
=\frac{ \kappa }{\lambda\sqrt{|1-(N\cdot z)^2|}}
=\frac{ \kappa \sqrt{||z|^2-1|}}{\|\wh{z}\|\sqrt{|1-(N\cdot z)^2|}},\]
where $\kappa$ denotes the curvature of the plane curve $\partial\Omega \subset \mathbb{R}^2$. We may write
\[\nu^\Si = a \dot{z}+b \frac{\wh{z}}{\|\wh{z}\|},\]
where (as $\operatorname{sign}(|x|^2)=-\sigma$),
\[-S\|\wh{z}\|^{-1} = a\ip{\dot{z}}{\frac{\wh{z}}{\|\wh{z}\|}}-\sigma b\]
and so
\begin{align*}
|\nu^\Si|^2 &= a^2 + 2ab\ip{\dot{z}}{\frac{\wh{z}}{\|\wh{z}\|}}-\sigma b^2
= a^2\br{1+\si \ip{\dot{z}}{\frac{\wh{z}}{\|\wh{z}\|}}^2}-\si S^2\|\wh{z}\|^{-2}. 
\end{align*}
We obtain
\[a^2 = \frac{|\nu^\Si|^2+\si S^2\|\wh{z}\|^{-2}}{1+\si\ip{\dot{z}}{\frac{\wh{z}}{\|\wh{z}\|}}^2}= \frac{\si (-\|\nu^\Si\|^2+S^2\|\wh{z}\|^{-2})}{1+\si\ip{\dot{z}}{\frac{\wh{z}}{\|\wh{z}\|}}^2},\]
where we note that, by the Cauchy--Schwarz inequality, the denominator is always positive and depends only on $\partial \Omega$. As $\frac{\wh{z}}{\|\wh{z}\|}$ is a zero eigenvector (as it is in direction $\frac{\partial}{\partial \lambda}$) we immediately see that
\begin{flalign}
\wh{h}(\nu^\Si, \nu^\Si)|_x 
&= \si (S^2\|\wh{z}\|^{-2}-\|\nu^\Si\|^2)\frac{\sqrt{||z|^2-1|} \kappa}{\|\wh{z}\|\br{1+\si\ip{\dot{z}}{\frac{\wh{z}}{\|\wh{z}\|}}^2}\sqrt{|1-\ip{N}{z}^2|}}\nonumber\\
&=:\frac{\si}{\|\wh{z}\|} (S^2\|\wh{z}\|^{-2}-\|\nu^\Si\|^2) q(\wh{z}),\label{eq:hathnunu}
\end{flalign}
where $q\cn\Sigma\ra \mathbb{R}$ is a bounded positive function with positive lower bound.

\subsection*{Mean curvature flow of capillary surfaces in cone coordinates}

In this paper, we are interested in the motion of $\al$-capillary surfaces in the cone $\cC$, as described in \eqref{eq:MCFBC}, where $T$ is the largest time, such that the flow exists as a flow of spacelike and smooth surfaces.
The use of $(\del_{t}x)^{\perp}$ ensures that this set of equations is geometric, i.e. invariant under time-dependent reparametrisations of $\Om$. If we fix a parametrisation by using cone coordinates, then
\eq{x(\xi,t) = u(\xi,t)(\xi + e_{3})=e^{\rho(\xi,t)}(\xi + e_{3})} evolves
by mean curvature flow if
\eq{H = -\ip{\del_{t}x(\xi,t)}{\nu} = -\del_{t}\rho\ip{x(\xi,t)}{\nu} = \fr{e^\rho \del_{t}\rho}{\rt{(1+\xi\cdot D\rho)^{2}-\abs{D\rho}^{2}}} }
and hence
\eq{\del_{t}\rho = g^{ij}(\xi,\rho, D\rho)D_{ij}\rho - g^{ij}(\xi,\rho, D\rho)D_{i}\rho D_{j}\rho .} 
In particular, we note that
\eq{\del_{t}\rho = \fr{H}{S}. \label{eq:rhot}}

We define a ``rescaled'' $\rho$ by
\eq{\wt{\rho} := \rho - \frac 1 2 \log(1+t) \label{eq:tilderhodef}} and a modified time function 
\eq{\tau := \log(1+t).\label{eq:timefunctiondef}}
Then \eqref{eq:MCFBC} is equivalent to 
\begin{equation}
\begin{cases}
\wt{\rho}_\tau = e^{-2\wt{\rho}}a^{ij}(\xi, D\wt{\rho})D_{ij} \wt{\rho} - e^{-2\wt{\rho}} a^{ij}(\xi,  D\wt{\rho}) D_i \wt{\rho} D_j \wt{\rho}-\frac 1 2& \text{ on } \Omega\times[0,\wt{T})\\
b(\xi, D\wt{\rho})=0&\text{ on } \partial \Omega\times[0,\wt{T})\\
\wt{\rho}(0, \cdot) = \wt{\rho}_0(\cdot) & \text{ on } \Omega,
\end{cases}\label{eq:MCFrescaledrho}
\end{equation}
where, as $D\wt{\rho} = D\rho$ and using the explicit formula in \eqref{eq:ginverse},
\[a^{ij}(\xi, D\wt{\rho}) = e^{2\rho} g^{ij}(x,\rho,D\rho)\]
and, using \eqref{nu} and \eqref{mu}, 
\[b(\xi, p) = \frac{p\cdot N-N\cdot\xi(1+\xi\cdot p)}{\sqrt{\abs{(N\cdot \xi)^2-1}}\sqrt{(1+\xi\cdot p)^2-|p|^2}}+\alpha.\]
In the above $\wt{T}$ may be taken to be the largest time such that the flow is spacelike, smooth and graphical. 
\subsection*{Conditions for uniform parabolicity and obliqueness}

\begin{lemma}\label{lem:parabolicz}
Suppose that there exist constants $c_{Su^{-1}}, C_v>0$ such that $Su^{-1}>c_{Su^{-1}}$ and $v<C_v$ for all $t\in[0,T)$. Then there exist constants $0<c=c(c_{Su^{-1}},C_v,\Omega) < C=C(c_{Su^{-1}},C_v,\Omega)$ such that the eigenvalues $\lambda_i$ of $a^{ij}=e^{2\rho}g^{ij}$ satisfy $c<\lambda<C$.
\end{lemma}

\begin{proof}
We note that in terms of $\rho$,
$$
Su^{-1}=\frac{1}{\sqrt{(1+D\rho\cdot\xi)^2-|D\rho|^2}},\qquad  1-\frac{1}{v^2} = \frac{|D\rho|^2}{(1+D\rho\cdot\xi)^2},\qquad
\frac{v}{Su^{-1} }=1+D\rho\cdot\xi.$$ 
In particular,
\[|D\rho|^2 = (1-v^{-2})\frac{v^2}{(Su^{-1})^2}\leq \frac{v^2}{(Su^{-1})^2}\leq C_v^2c_{Su^{-1}}^{-2},\]
and
\[\frac{|D\rho|^2}{(1+D\rho\cdot\xi)^2-|D\rho|^2}=(1-v^{-2})\frac{(1+D\rho\cdot\xi)^2}{(1+D\rho\cdot\xi)^2-|D\rho|^2}=v^2-1.\]
For any unit vector $w$ with respect to the Euclidean metric on $\Omega$, we have
\[e^{-2\rho}w^ig_{ij}w^j\leq 1+2|\xi||D\rho| +|\xi|^2|D\rho|^2<C_1(c_{Su^{-1}},C_v),\]
and
\begin{flalign*}e^{2\rho}w_ig^{ij}w_j &\leq 1+\frac{|D\rho|^2+|\xi|^2|D\rho|^2+2(1+D\rho\cdot \xi)|D\rho||\xi|}{(1+D\rho\cdot\xi)^2-|D\rho|^2}\\&\leq 1+(1+|\xi|^2)(v^2-1)+2v\sqrt{v^2-1}|\xi|\\
&\leq C_2(c_{Su^{-1}},C_v).
\end{flalign*}
This implies that the eigenvalues are strictly bounded away from zero and infinity. As $g_{ij}>0$ at time zero, this is hence preserved with the claimed estimates being valid.
\end{proof}


\begin{lemma}\label{lem:obliquerho}
Let $M\sub \cC$ be $\al$-capillary, with radial graph function $u$. Then
\eq{N^{i}\fr{\del b}{\del p^{i}}\big|_{p=D\rho} = \rt{\abs{(N\cdot\xi)^{2}-1}}\br{(b(\xi,D\rho)-\al)^{2}+\si }Su^{-1}.}
Therefore $b$ is uniformly oblique given $Su^{-1}>c_{Su^{-1}}>0$.
\end{lemma}

\pf{

We calculate
\eq{N^{i}\fr{\del b}{\del p^{i}} &= \fr{1 - (N\cdot\xi)^{2}}{\sqrt{\abs{(N\cdot \xi)^2-1}}\sqrt{(1+\xi\cdot p)^2-|p|^2}}-(b-\al)\fr{(1+\xi\cdot p)N\cdot\xi-N\cdot p}{(1+\xi\cdot p)^2-|p|^2}\\
				&= \fr{\rt{\abs{(N\cdot\xi)^{2}-1}}}{\rt{(1+\xi\cdot p)^2-|p|^2}}\br{\br{b-\al}^{2}+\fr{1-(N\cdot\xi)^{2}}{\abs{1-(N\cdot\xi)^{2}}} }.}
We use $b=0$ to complete the proof.

}

We collect \autoref{lem:parabolicz} and \autoref{lem:obliquerho} in the following Proposition.
 \begin{prop}\label{prop:parabolicityobliqueness}
Suppose that there exists constants $C_v,c_{Su^{-1}},C_{\wt{\rho}},c_{\wt{\rho}}>0$ such that for all $t\in[0,T)$, we have
\[v<C_v,\qquad Su^{-1}>c_{Su^{-1}}, \qquad\text{ and }\qquad c_{\wt{\rho}}<\wt{\rho}<C_{\wt{\rho}}.\]
Then \eqref{eq:MCFrescaledrho} is uniformly parabolic and has a uniformly oblique boundary condition.
\end{prop}

\section{Boundary identities}
\label{sec:BoundaryIdentities}

We calculate the boundary properties of the flow. We note that  \autoref{lem:1stspace} and \autoref{lem:1sttime} hold for any manifold $\Sigma$ as capillary boundary condition.
\begin{lemma}[First space derivatives]\label{lem:1stspace} 
For any $Y \in T\partial M_t$,
\[0= h(Y, \mu^\top)+\wh{h}(Y, \nu^\Sigma).\]
\end{lemma}
\begin{proof}
We have
\[0=Y(\ip{\nu}{\mu}) =\ip{\ov \n_Y\nu}{\mu}+\ip{\nu}{\ov \n_Y\mu} = h(Y, \mu^\top)+\wh{h}(Y, \nu^\Sigma). \]
\end{proof}

\begin{lemma}[First time derivative]\label{lem:1sttime}
On $\del M_t$,
\[\n_{\mu^\top} H =H \left[-\sigma\wh{h}\left(\frac{\nu^\Sigma}{\|\nu^\Sigma\|},\frac{\nu^\Sigma}{\|\nu^\Sigma\|}\right)-\al  h\left(\frac{\mu^{\top}}{|\mu^{\top}|},\frac{\mu^{\top}}{|\mu^{\top}|}\right)\right].\]
\end{lemma}
\begin{proof}
We consider a reparametrisation of mean curvature flow $y\cn B_1(0)\ra \mathbb{R}^{3}_1$ such that $\operatorname{Im}(y|_{\partial B_1(0)})\subset \Sigma$. We have on $\partial B_1(0)$, \eq{\br{\frac{dy}{dt}}^{\top} = \beta \mu^\top + V} where $V$ is a vector field in $T\partial M_t$ and $\beta$ is some function. By reparametrising near the boundary we may choose $V$ to be zero.

We know that $y\in \Sigma$ for all time and so 
\[0=\ip{H\nu +\beta \mu^\top}{\mu} = -H\al+\beta(\al^2+\sigma),\] 
and at the boundary
\[\frac{dy}{dt} = H\nu +\frac{\al}{\al^2+\si} H\mu^\top = \si\frac{H}{\al^2+\si}\nu^\Sigma. \]
Under this velocity we have 
\[\ip{\frac{d}{dt}\nu(y)}{y_i} = -\ip{\nu}{(H\nu +\frac{\al}{\al^2+\si} H\mu^\top)_i}=H_i+\frac{\al}{\al^2+\si}H h(y_i, \mu^\top).\]
Applying this we have
\[0=\frac{d}{dt}\ip{\nu(y)}{\mu(y)}=\n_{\mu^\top} H +\al H h\left(\frac{\mu^{\top}}{|\mu^{\top}|},\frac{\mu^{\top}}{|\mu^{\top}|}\right)+\sigma H \wh{h}\left(\frac{\nu^\Sigma}{\|\nu^\Sigma\|},\frac{\nu^\Sigma}{\|\nu^\Sigma\|}\right).\]
\end{proof}

The following identities will be useful to us later.
\begin{lemma}\label{lem:Sboundary}
On $\partial M_t$ we have 
\[\nabla_{\mu^\top}S=S\left[-\si\widehat{h}\left(\frac{\nu^\Sigma}{\|\nu^\Sigma\|},\frac{\nu^\Sigma}{\|\nu^\Sigma\|}\right)-\alpha h\left(\frac{\mu^\top}{|\mu^\top|},\frac{\mu^\top}{|\mu^\top|}\right)\right]. \]
\end{lemma}
\begin{proof}
We use that $\ip{x}{\mu} = 0$ and represent $x$ using two orthonormal bases:
\[x = x^{M\cap\Si} + S\nu+ \ip{x}{\frac{\mu^\top}{|\mu^\top|}}\frac{\mu^\top}{|\mu^\top|} = x^{M\cap\Sigma} -\sigma \ip{x}{\frac{\nu^\Sigma}{\|\nu^\Sigma\|}}\frac{\nu^\Sigma}{\|\nu^\Sigma\|}.\]
Rewriting and rearranging the second equality,
\begin{equation}x-\sigma\frac{S}{\|\nu^\Sigma\|^2}\nu^\Sigma = x^{M\cap\Sigma}.\label{eq:xdecomp}\end{equation}
We have
\begin{align}\begin{split}
\nabla_{\mu^\top}S &= -h\left(x^\top,\mu^\top\right)\\
                &=-h\left(x^{M\cap\Sigma} + \ip{x^\top}{\frac{\mu^\top}{|\mu^\top|}}\frac{\mu^\top}{|\mu^\top|},\mu^\top\right)\\
                &=\widehat{h}(x^{M\cap\Sigma},\nu^{\Sigma}) - \ip{x^\top}{\mu^\top}h\left(\frac{\mu^\top}{|\mu^\top|},\frac{\mu^\top}{|\mu^\top|}\right)\\
                &=-\sigma S\widehat{h}\left(\frac{\nu^\Sigma}{\|\nu^\Sigma\|},\frac{\nu^\Sigma}{\|\nu^\Sigma\|}\right) - \ip{x}{\mu-\alpha\nu}h\left(\frac{\mu^\top}{|\mu^\top|},\frac{\mu^\top}{|\mu^\top|}\right)\\
              &=S\left(-\sigma \widehat{h}\left(\frac{\nu^\Sigma}{\|\nu^\Sigma\|},\frac{\nu^\Sigma}{\|\nu^\Sigma\|}\right)-\alpha h\left(\frac{\mu^\top}{|\mu^\top|},\frac{\mu^\top}{|\mu^\top|}\right)\right),
\end{split}\end{align}
where we used \autoref{lem:1stspace} on the third line and that $\wh{h}$ has a zero eigenvalue in the $x$ direction and \eqref{eq:xdecomp} on the fourth line.
\end{proof}

\begin{lemma}\label{lem:vboundary}
On $\partial M_t$ the following holds:
\[\nabla_{\mu^\top}v=\wh{h}(e_{3}^{M\cap \Si}, \nu^\Sigma)-(\ip{e_{3}}{\mu}+\al v)h\left(\frac{\mu^\top}{|\mu^\top|},\frac{\mu^\top}{|\mu^\top|}\right).\]
\end{lemma}
\begin{proof}
We have 
\begin{flalign*}
e_{3}^\top &= e_{3}^{M\cap\Si} +\ip{e_{3}^\top}{\mu^\top}\mu^{\top}(\al^2+\si)^{-1} =e_{3}^{M\cap\Si} +(\ip{e_{3}}{\mu}+\al v)\mu^{\top}(\al^2+\si)^{-1}, 
\end{flalign*}
so using \autoref{lem:1stspace}
\begin{flalign*}
\nabla_{\mu^\top}v&=-h(e_{3}^\top, \mu^\top)=\wh{h}(e_{3}^{M\cap\Si}, \nu^\Sigma)-(\ip{e_{3}}{\mu}+\al v)h\left(\frac{\mu^\top}{|\mu^\top|},\frac{\mu^\top}{|\mu^\top|}\right).\\
\end{flalign*}
\end{proof}

\section{Initial estimates}
\label{sec:InitialEstimates}
In this section we will prove a priori estimates for the flow \eqref{eq:MCFBC} which hold for nondegenerate boundaries $\Si$, namely either of $\sigma = \pm 1$. Recall that $T$ is the largest time of smooth and spacelike existence.
As is standard, we will use time dependent rescalings to understand the asymptotics of this solution. We define
\[\wt{M}_t := (1+t)^{-1/2}M_t \; \subset \; \bbR^{3}_1,\]
and  we will typically consider $\wt{M}_t$ in terms of the rescaled time coordinate $\tau = \log (1+t)$. We will add tilde's to all geometric quantities calculated on $\wt{M}_\tau$. We immediately see that
\[\left(\partial_\tau \wt{x}\right)^\perp 
=\wt{H}\wt{\nu}-\frac 1 2 \widetilde{x}^\perp. \]

We first observe the following. 
\begin{lemma}[Uniform spacelikeness of $M_t$ for Riemannian $\Si$]\label{lem:presspacelike} 
Under the assumptions of \autoref{thm:maintheorem} and in the case $\Sigma$ is Riemannian (i.e. $\si = -1$), there exists a constant $C=C(\alpha, \Sigma)$ such that for all $t\in[0,T)$, 
\[v\leq \max\{ \sup_{M_0} v, C\}=:C_v.\]
\end{lemma}
\begin{proof}

We first bound $v$ on $\partial M_t$. We write $\mu = a e_{3}+V$, $\nu = b e_{3}+W$ where $V,W \perp e_3$ and $a, b>0$. Then $|W|^2 = b^2-1$,  $|V|^2 = a^2-1$ and $\alpha = ab-\ip{V}{W}$. By Cauchy-Schwarz,
\[ab \leq \alpha+|V||W| = \alpha+ \sqrt{a^2-1}\sqrt{b^2-1}.\]
Squaring both sides and rearranging gives
\[a^2+b^2\leq \al^2 +2\al\sqrt{a^2-1}\sqrt{b^2-1}+1\leq \al^2 +2\al ab+1\leq \al^2+\frac{b^2}{2}+2\al^2a^2+1,\]
where we used Young's inequality to get the last estimate. Finally we see that
\[v = b \leq \sqrt{2\al^2+2(2\al^2-1)\ip{\mu}{e_{3}}^2+2}.\]
By uniform timelikeness of $\mu$, $\abs{\ip{\mu}{e_{3}}}<C(\Sigma)$ and so $v\leq C(\alpha, \Sigma)$ on $ \partial M_t$.


As to the interior, the evolution of $v$ is given by
\[\ho v = -|A|^2v,\]
we apply the weak maximum principle for an interior bound. This completes the proof.
\end{proof}


\begin{lemma}[Preservation of mean convexity of $M_t$]\label{lem:presmeanconvex}
Under the assumptions of \autoref{thm:maintheorem}, for all $t\in[0,T)$, $M_t$ remains strictly mean convex.
\end{lemma}
\begin{proof}
In the interior of $M_t$, by \cite[Proposition 2.6(i)]{Ecker:/1997},
\[\ho H = -H|A|^2\]
and as $H>0$ on $M_0$, any interior point $p$ with $H(p)=0$ would contradict the strong maximum principle. Let $t_1$ be the first time at which the $\min_{M_t} H=0$. Then there exists a $p\in\partial M_t$ such that $H(p,t_1)=0$, and so, by \autoref{lem:1sttime}, $\n_{\mu^\top} H=0$. This contradicts the parabolic Hopf Lemma so no such $t_1$ exists.
\end{proof}

\begin{lemma}[Preservation of graphicality of $M_t$]\label{lem:presgraph} 
Under the assumptions of \autoref{thm:maintheorem}, for all $t\in[0,T)$, $M_t$ is graphical, i.e. $S> 0$.
\end{lemma}
\begin{proof}
In the case that $\Si$ is Lorentzian, $S|x|^{-1}\geq 1$ and so the above follows automatically. If $\Si$ is Riemannian (i.e. $\si=-1$), then by \autoref{lem:presmeanconvex}, $H>0$ is preserved and so, using \cite[Lemma~4.3]{Lambert2014},
\[\ho S = -|A|^2S+2H\geq -|A|^2 S.\]
By the strong maximum principle, there can therefore be no interior minima of $S$ with $S=0$.

As in \autoref{lem:presmeanconvex}, a boundary minimum with $S=0$ would contradict the Hopf Lemma due to the form of the boundary derivative in \autoref{lem:Sboundary}. Therefore graphicality is preserved.
\end{proof}

\begin{lemma}[Bound on speed]\label{lem:HSbound} 
Under the assumptions of \autoref{thm:maintheorem}, while $t\in[0,T)$, there exist constants $0<c_{\frac{S}{H}}<C_{\frac{S}{H}}$ depending only on initial data, such that
\[2(c_{\frac{S}{H}}+t)\leq \frac{S}{H}\leq 2(C_{\frac{S}{H}}+t).  \]
\end{lemma}
\begin{proof}
We have
\begin{flalign*}
\ho \left(\frac{S}{H}-2t\right)&= \frac{S}{H}\left[\frac{1}{S} \ho S - \frac 1 H \ho H\right]+2\ip{\frac{\n H}{H} }{\n \frac{S}{H}}-2\\
&=2\ip{\frac{\n H}{H} }{\n \frac{S}{H}}.
\end{flalign*}
By \autoref{lem:1sttime} and \autoref{lem:Sboundary}, $H$ and $S$ satisfy the same linear boundary condition, and so 
\[\n_{\mu^\top} \left(\frac{S}{H}-2t\right)=0.\]
We may apply the maximum principle to $\frac{S}{H}-2t$ to get the claimed estimates with $C_{\frac{S}{H}}=\frac 1 2\sup_{M_0} \frac{S}{H}$ and $c_{\frac{S}{H}}=\frac 1 2\inf_{M_0} \frac{S}{H}$.
\end{proof}

\begin{lemma}[Bound on rescaled flow]\label{lem:tildeubound} 
Under the assumptions of \autoref{thm:maintheorem}, there exist constants $c_{\wt{\rho}}<C_{\wt{\rho}}$ depending only on $M_0$ such that while $t\in[0,T)$,
\[c_{\wt{\rho}}\leq \wt{\rho}\leq C_{\wt{\rho}}.\]
\end{lemma}
\begin{proof}  
Recall that by \eqref{eq:rhot}, $\rho_t=H/S.$
By applying \autoref{lem:HSbound} and using the definition of $\wt{\rho}$, \eqref{eq:tilderhodef}, we see that
\[\frac{1-C_{\frac{S}{H}}}{2(C_{\frac{S}{H}}+t)(1+t)}\leq \wt{\rho}_t\leq \frac{1-c_{\frac{S}{H}}}{2(c_{\frac{S}{H}}+t)(1+t)}.\]
The Lemma now follows by integration in time.
\end{proof}

We also note that for $\xi\in \del\Om$, there exist constants $0<c_{\|x\|^2u^{-2}}<C_{\|x\|^2u^{-2}}$ depending only on $M_0$ with
\begin{equation}c_{\|x\|^2u^{-2}}u^2(\xi,t)\leq\|x(\xi,t)\|^2\equiv u^2(\xi,t)\abs{\abs{\xi}^2 - 1}\leq C_{\|x\|^2u^{-2}}u^2(\xi,t)\label{eq:xuLipschitzequiv}.\end{equation} 
Applying \autoref{lem:tildeubound} to this we see that there exist constants $0<c_{\frac{\|x\|}{1+t}}<C_{\frac{\|x\|}{1+t}}$ such that for the position vector $x\in \partial M_t$,
\eq{c_{\frac{\|x\|^2}{1+t}}< \frac{\|x\|^2}{1+t}<C_{\frac{\|x\|^2}{1+t}}.\label{eq:xtotime}}

\section{The graphicality estimate in case \texorpdfstring{$\Si$}{Sigma} is Riemannian}
\label{sec:SpacelikeCone}
In this section we work towards strict graphicality i.e. improving  \autoref{lem:presgraph} to a strictly positive lower bound on $S$   in case   $\Si$ is a Riemannian surface.  We start by noting some useful linear algebra. 

By Cauchy-Schwarz, at any boundary point we have  \[S=-\ip{\nu^\Sigma}{x}\leq|x|\sqrt{\al^2-1}.\] However, if $S<|x|\sqrt{\al^2-1}$ (i.e.  $x$ and $\nu^\Sigma$ are not parallel) then $\left\{x, \nu^\Si\right\}$ is a basis for $T \Sigma$. In the following computations we will assume that we are in this case. Therefore we see that for any $Z\in T\partial M_t$, 
\[Z =\frac{S\ip{Z}{x}}{|x|^2(\al^2-1)-S^2}\nu^\Sigma+(\al^2-1)\frac{\ip{Z}{x}}{|x|^2(\al^2-1)-S^2}x.\] 
As $x$ is a zero eigenvector of $\wh{h}$, we may now apply the above identity to get the useful formula
\begin{equation}\wh{h}(Z, \nu^\Sigma)=\frac{(\al^2-1)S\ip{Z}{x}}{(\al^2-1)|x|^2-S^2}\wh{h}\left(\frac{\nu^\Sigma}{|\nu^\Sigma|}, \frac{\nu^\Sigma}{|\nu^\Sigma|}\right).\label{eq:boundarycurvatureprojection}
\end{equation}

\begin{lemma}
 Suppose the assumptions of \autoref{thm:maintheorem} hold and that $\Si$ is Riemannian. Then there exist constants $c, C>0$ depending on $\al, \Sigma$, and on $C_v, C_{\frac{S}{H}}, c_{\wt\rho}, C_{\wt\rho}$ in \autoref{lem:presspacelike}, \autoref{lem:HSbound} and \autoref{lem:tildeubound}, such that the following holds. For any $x\in \partial M_t$ at which $\frac S v$ attains a minimum of value $\frac S v<c|x|$, we have
\[\n_{\mu^\top}\frac Sv 
\geq \frac Sv\left[(1-C\frac S v|x|^{-1})\widehat{h}\left(\frac{\nu^\Sigma}{|\nu^\Sigma|},\frac{\nu^\Sigma}{|\nu^\Sigma|}\right)
-C\frac Sv|x|^{-2}\right] .\]\label{prop:boundaryderivcalc}
\end{lemma}
 
\begin{proof}
Using \autoref{lem:presspacelike}, we choose $c<C_v^{-1}\sqrt{\al^2-1}$ to be determined later so that $S<|x|\sqrt{\al^2-1}$. Therefore, $x$ and $\nu^{\Si}$ are linearly independent, and we may apply the linear algebraic computations above. In particular,
\begin{flalign*}
\n_{\mu^\top} \frac{S}{v} &= \frac{S}{v}\left[S^{-1}\n_{\mu^\top} S-v^{-1}\n_{\mu^\top} v\right] \\
&=\frac{S}{v}\left[\widehat{h}\left(\frac{\nu^\Sigma}{|\nu^\Sigma|},\frac{\nu^\Sigma}{|\nu^\Sigma|}\right)-\alpha h\left(\frac{\mu^\top}{|\mu^\top|},\frac{\mu^\top}{|\mu^\top|}\right)
-v^{-1}\wh{h}(e_{3}^{\Sigma\cap M}, \nu^\Sigma)\right.\\&\qquad\qquad\qquad\qquad\qquad\qquad\left.+(v^{-1}\ip{e_{3}}{\mu}+\al )h\left(\frac{\mu^\top}{|\mu^\top|},\frac{\mu^\top}{|\mu^\top|}\right)\right] \\
&=\frac{S}{v}\left[\left(1-v^{-1}\frac{(\al^2-1)S\ip{e_{3}^{\Si\cap M}}{x}}{(\al^2-1)|x|^2-S^2}\right)\widehat{h}\left(\frac{\nu^\Sigma}{|\nu^\Sigma|},\frac{\nu^\Sigma}{|\nu^\Sigma|}\right)\right.\\&\qquad\qquad\qquad\qquad\qquad\qquad\left.
-v^{-1}|\ip{e_{3}}{\mu}|h\left(\frac{\mu^\top}{|\mu^\top|},\frac{\mu^\top}{|\mu^\top|}\right)\right], 
\end{flalign*}
where we used \eqref{eq:boundarycurvatureprojection}.
By assumption, $Sv^{-1}$ attains a global minimum over $M_t$ at $x\in \partial M_t$, and so $Sv^{-1}$ also attains its minimum over $\partial M_t$ at $x$. Therefore,  writing a unit tangent to $\partial M_t$ as $\gamma$, we have

\eq{
0=\n_\ga \frac Sv &= \frac Sv\left[-h(S^{-1}x^{\top}-v^{-1}e^\top_{3}, \ga)\right]\\   
&= \frac Sv\left[-\ip{S^{-1}x-v^{-1}e_{3}}{\ga}h(\ga, \ga)+\ip{S^{-1}x-v^{-1}e_{3}}{\frac{\mu^\top}{|\mu^\top|}}\wh{h}\left(\frac{\nu^\Sigma}{|\nu^\Sigma|}, \ga\right)\right].\label{eq:fromtangentialderiv}
}

Note that
\begin{flalign*}
\ip{S^{-1}x-v^{-1}e_{3}}{\mu^\top}&=\ip{S^{-1}x-v^{-1}e_{3}}{\mu}=-v^{-1}\ip{\mu}{e_{3}}=v^{-1}|\ip{\mu}{e_{3}}|>0,
\end{flalign*}
while
\begin{flalign*}
|S^{-1}x-v^{-1}e_{3}|^2 &= S^{-2}|x|^2-2v^{-1}S^{-1}\ip{x}{e_{3}}-v^{-2}=S^{-2}|x|^2+2v^{-1}uS^{-1}-v^{-2}.
\end{flalign*}
Therefore, from the previous two equations we have 
\begin{flalign*}
\ip{S^{-1}x-v^{-1}e_{3}}{\ga}^2 &= S^{-2}|x|^2+2v^{-1}uS^{-1}-v^{-2}-v^{-2}\ip{\mu}{e_{3}}^2(\al^2-1)^{-1}\\
&\geq S^{-2}(\check{c}|x|^2-\check{C}v^{-2}S^2)
\end{flalign*}
for some $0<\check{c}$ and $1<\check{C}$ depending only on $\Sigma$ and $\alpha$. We also modify these constants for future estimates (by making $\check{c}$ smaller  and $\check{C}$ larger depending on the spacelikeness estimate in \autoref{lem:presspacelike}) so that, as a result of the upper bound on $S$ as above,
\[|x|^2(\al^2-1)-S^2>\check{c}|x|^2-\check{C}v^{-2}S^2 > 0.\] 

Applying \eqref{eq:boundarycurvatureprojection}, we also have
\begin{flalign*}\ip{S^{-1}x-v^{-1}e_{3}}{\frac{\mu^\top}{|\mu^\top|}}&\wh{h}\br{\frac{\nu^\Sigma}{|\nu^\Sigma|}, \ga}
=\frac{S\ip{\gamma}{x}v^{-1}|\ip{e_{3}}{\mu}|}{(\al^2-1)|x|^2-S^2}\widehat{h}\left(\frac{\nu^\Sigma}{|\nu^\Sigma|},\frac{\nu^\Sigma}{|\nu^\Sigma|}\right).\end{flalign*}

Using \eqref{eq:fromtangentialderiv} at $p$, along with the above estimates and identities we see that
\begin{flalign*}
h\left(\frac{\mu^\top}{|\mu^\top|},\frac{\mu^\top}{|\mu^\top|}\right)&=H-h(\ga,\ga)\\
&\leq \frac{S}{2c_{\frac{S}{H}}+2t}-\ip{S^{-1}x-v^{-1}e_{3}}{\ga}^{-1}\frac{S\ip{\gamma}{x}v^{-1}|\ip{e_{3}}{\mu}|}{|x|^2(\al^2-1)-S^2}\widehat{h}\left(\frac{\nu^\Sigma}{|\nu^\Sigma|},\frac{\nu^\Sigma}{|\nu^\Sigma|}\right) \\
&\leq \frac{S}{2c_{\frac{S}{H}}+2t} +\frac{S^2v^{-1}|\ip{\gamma}{x}\ip{e_{3}}{\mu}|}{(\check{c}|x|^2-\check{C}v^{-2}S^2)^\frac 3 2}\widehat{h}\left(\frac{\nu^\Sigma}{|\nu^\Sigma|},\frac{\nu^\Sigma}{|\nu^\Sigma|}\right).
\end{flalign*}

Therefore,
\eq{
\n_{\mu^\top} \frac{S}{v} &\geq\frac{S}{v}\left[\left(1-v^{-1}\frac{S(\al^2-1)|\ip{e_{3}^{\Si \cap M}}{x}|}{\check{c}|x|^2-\check{C}v^{-2}S^2}-\frac{S^2v^{-2}|\ip{\gamma}{x}|\ip{e_{3}}{\mu}^2}{(\check{c}|x|^2-\check{C}v^{-2}S^2)^\frac 3 2}\right)\widehat{h}\left(\frac{\nu^\Sigma}{|\nu^\Sigma|},\frac{\nu^\Sigma}{|\nu^\Sigma|}\right)
\right.\\
&\qquad\qquad\qquad\qquad\qquad\qquad \left.-v^{-1}|\ip{e_{3}}{\mu}|\frac{S}{2c_{\frac{S}{H}}+2t}\right].
}
We now choose 
\eq{c= \min\left\{\rt{\frac{\check{c}}{2\check{C}}}, C_v^{-1}\sqrt{\al^2-1}\right\},} and using this bound, we may find a $C$ depending only on $\alpha$, $\Sigma$ and estimates in \autoref{lem:presspacelike}, \ref{lem:HSbound} and \ref{lem:tildeubound} so that
\begin{flalign*}
\n_{\mu^\top} \frac{S}{v} 
&\geq \frac{S}{v}\left[(1-C\frac Sv|x|^{-1})\widehat{h}\left(\frac{\nu^\Sigma}{|\nu^\Sigma|},\frac{\nu^\Sigma}{|\nu^\Sigma|}\right)
-C\frac Sv|x|^{-2}\right].
\end{flalign*}

\end{proof}

\begin{prop}[Uniform graphicality estimate]\label{prop:UniformGraphSpacelikeCone}
If the assumptions of \autoref{thm:maintheorem} hold and $\Sigma$ is Riemannian, then there exists $c>0$ depending only on $\alpha$, $\Sigma$ and the estimates in \autoref{lem:presspacelike}, \autoref{lem:HSbound} and \autoref{lem:tildeubound} such that for all $t\in [0,T)$,
\[\inf_{M_t} \frac S{vu} \geq c.\]
\end{prop}
\begin{proof} Throughout this proof we let $c_i, C_i>0$, $i\in\mathbb{N}$, be constants depending on  $\alpha$, $\Sigma$ and the estimates in \autoref{lem:presspacelike}, \autoref{lem:HSbound} and \autoref{lem:tildeubound}. We consider the function $\frac Sv$. By \eqref{eq:hathnunu} and \autoref{lem:presspacelike} there exists $C_1$ and $c_\Sigma=c_\Sigma(\Sigma)$ such that,  at the boundary,
\[\widehat{h}\left(\frac{\nu^\Sigma}{|\nu^\Sigma|},\frac{\nu^\Sigma}{|\nu^\Sigma|}\right) =q(x) |x|^{-1}(1-S^2|x|^{-2}(\al^2-1)^{-1})>c_\Sigma|x|^{-1}-C_1S^2v^{-2}|x|^{-3},\]
where we used the strict convexity of $\Si$ for the lower bound on $q$. While $\frac Sv<\sqrt{\frac{c_\Si}{2C_1}}|x|$ we therefore have
\[\widehat{h}\left(\frac{\nu^\Sigma}{|\nu^\Sigma|},\frac{\nu^\Sigma}{|\nu^\Sigma|}\right) >\frac 1 2 c_\Sigma|x|^{-1}.\]
Now, for $c$ as in \autoref{prop:boundaryderivcalc}, while $\frac Sv<\min\left\{c,\sqrt{\frac{c_\Si}{2C_1}}\right\}|x|$, there holds
\[\n_{\mu^\top} \frac Sv \geq \frac Sv|x|^{-1}\left[\frac{1}{4}c_\Sigma -C_2\frac S v|x|^{-1}\right].\]
We see that as $\mu^\top$ is outward pointing, a boundary minimum is not possible for 
\eq{\frac S {v|x|}<\min\left(c,\sqrt{\frac{c_\Si}{2C_1}}, \frac 1 4c_\Sigma C_2^{-1}\right)=:c_2.} Hence, using \eqref{eq:xtotime} and \autoref{lem:tildeubound}, there exists a $c_3>0$ such that at a boundary minimum there holds
\[\frac Sv\geq c_3\sqrt{1+t}.\]

Next in the interior, using \autoref{lem:HSbound}, 
\begin{flalign*}
\ho \frac Sv &= \frac{S}{v}\left[S^{-1} \ho S-v^{-1}\ho v\right]+2\ip{v^{-1}\n v}{\n \frac S v}\\
&= 2\frac{H}{v}+2\ip{v^{-1}\n v}{\n \frac Sv}\\
&\geq \frac{1}{( C_{\frac{S}{H}}+t)}\frac S v+2\ip{v^{-1}\n v}{\n \frac Sv}.
\end{flalign*}
Therefore
\begin{flalign*}
\ho \frac Sv( C_{\frac{S}{H}}+t)^{- \frac{1}{2} } &\geq \frac{1}{2(C_{\frac{S}{H}}+t)^\frac 3 2}\frac Sv+2\ip{v^{-1}\n v}{\n\br{\frac S v(C_{\frac{S}{H}}+t)^{- \frac 1 2}}}.
\end{flalign*}
Applying the weak maximum principle yields 
$\frac Sv( C_{\frac{S}{H}}+t)^{- \frac{1}{2} }>c_4(M_0, c_3, C_{\frac{S}{H}})$.
 The Lemma now follows by \autoref{lem:tildeubound}. \end{proof}
\begin{rem}
    We note that the evolution equation above at first glance appears to allow for stronger lower bounds. However, such a global estimate cannot hold since smaller boundary minima would then not be disallowed by the above boundary estimates.
\end{rem}

\section{A uniform spacelikeness estimate in case \texorpdfstring{$\Si$}{Sigma} is Lorentzian}
\label{sec:LorentzCone}
We now obtain a uniform upper estimate on $S$ in case $\Si$ Lorentzian, namely $\si=1$. 
We have \[S^2=\|x\|^2\|\nu^\Si\|^2\ip{\frac{x}{\|x\|}}{\frac{\nu^\Si}{\|\nu^\Si\|}}^2\geq \|x\|^2(\al^2+1),\] as the inner product of unit timelike vectors has modulus greater than or equal to 1. As in the previous section, we only deal with the case $S^2> \|x\|^2(\al^2+1)$ as we are aiming to bound $S$ from above. In this case $\left\{x,\nu^\Si\right\}$ is a basis of $T\Sigma$. As previously, for any $Z\in T\partial M_t$ we may calculate
\[Z=(\al^2+1)\frac{\ip{Z}{x}}{S^2-\|x\|^2(\al^2+1)} x -\frac{\ip{Z}{x}S}{S^2-\|x\|^2(\al^2+1)}\nu^\Sigma.\]
 We therefore have
\begin{equation}
\wh{h}(Z,\nu^\Sigma)=-\frac{\ip{Z}{x}(\al^2+1)S}{S^2-\|x\|^2(\al^2+1)}\wh{h}\left(\frac{\nu^\Sigma}{\|\nu^\Sigma\|},\frac{\nu^\Sigma}{\|\nu^\Sigma\|}\right).\label{eq:LorentzSigma2ff}
\end{equation}

\begin{lemma}
Suppose that the assumptions of \autoref{thm:maintheorem} hold and
 that $\Sigma$ is Lorentzian. Then for any $x\in\partial M_t$ at which $S$ attains its global maximum of value $S^2> \|x\|^2(\al^2+1)$, we have \label{prop:timelikeboundaryspacelike}
\begin{flalign*}
\nabla_{\mu^\top}S &= S\left(-\frac{(S^2-\|x\|^2)(\al^2+1)}{S^2-\|x\|^2(\al^2+1)}\wh{h}\left(\frac{\nu^\Sigma}{\|\nu^\Sigma\|},\frac{\nu^\Sigma}{\|\nu^\Sigma\|}\right)-\alpha H\right).\\
\end{flalign*}
\end{lemma}
\begin{proof}
We recall that, by \autoref{lem:Sboundary},  
\[\nabla_{\mu^\top}S=S\left(-\widehat{h}\left(\frac{\nu^\Sigma}{|\nu^\Sigma|},\frac{\nu^\Sigma}{|\nu^\Sigma|}\right)-\alpha h\left(\frac{\mu^\top}{|\mu^\top|},\frac{\mu^\top}{|\mu^\top|}\right)\right).\]
As $S$ attains its global maximum at $x\in \partial M_t$,  $S$ also attains its maximum over $\del M_t$. Therefore, for a unit vector $\gamma\in T_p\del M_t$, 
\[0=\n_\gamma S=-h(x^\top, \gamma) = -\ip{x^\top}{\gamma}h(\gamma, \gamma)+\ip{x^\top}{\frac{\mu^\top}{|\mu^\top|}}\wh{h}\br{\gamma, \frac{\nu^\Sigma}{\|\nu^\Sigma\|}}.\]
We know that $\ip{x^\top}{\gamma}=\ip{x}{\gamma}\neq 0$ (as otherwise $x=\beta \nu^\Si$ and so $S^2=\|x\|^2(\al^2+1)$). Therefore we may write
\begin{equation}{h}(\ga,\ga) = \frac{\ip{x^\top}{\mu^\top}}{\ip{x}{\ga}(\al^2+1)}\wh{h}(\gamma,\nu^\Sigma).\label{eq:justbdry2ff}\end{equation}
We have
\[\ip{x^\top}{\mu^\top}=S\alpha\]
and so
\[\ip{x^\top}{\gamma}^2=|x^\top|^2-\ip{x^\top}{\frac{\mu^\top}{|\mu^\top|}}^2=\frac{S^2}{\al^2+1}-\|x\|^2.\]
Therefore by \eqref{eq:LorentzSigma2ff} and \eqref{eq:justbdry2ff} we have 
\begin{flalign*}h\left(\frac{\mu^\top}{|\mu^\top|},\frac{\mu^\top}{|\mu^\top|}\right)=H-h(\ga,\ga)
&= H-\frac{S\al}{(\al^2+1)\ip{\gamma}{x}}\wh{h}\left(\gamma,\nu^\Sigma\right)\\&= H+\frac{S^2\al}{S^2-\|x\|^2(\al^2+1)}\wh{h}\left(\frac{\nu^\Sigma}{\|\nu^\Sigma\|},\frac{\nu^\Sigma}{\|\nu^\Sigma\|}\right)
\end{flalign*}
and so
\begin{flalign*}
\nabla_{\mu^\top}S
&= S\left(-\widehat{h}\left(\frac{\nu^\Sigma}{|\nu^\Sigma|},\frac{\nu^\Sigma}{|\nu^\Sigma|}\right)-\frac{S^2\al^2}{S^2-\|x\|^2(\al^2+1)}\wh{h}\left(\frac{\nu^\Sigma}{\|\nu^\Sigma\|},\frac{\nu^\Sigma}{\|\nu^\Sigma\|}\right)-\alpha H\right)\\
&= S\left(-\frac{(S^2-\|x\|^2)(\al^2+1)}{S^2-\|x\|^2(\al^2+1)}\wh{h}\left(\frac{\nu^\Sigma}{\|\nu^\Sigma\|},\frac{\nu^\Sigma}{\|\nu^\Sigma\|}\right)-\alpha H\right).
\end{flalign*}
\end{proof}

\begin{prop}[Preservation of uniform spacelikeness]\label{prop:uniformspacelikeness}
 Suppose that the assumptions of \autoref{thm:maintheorem} hold and that $\Sigma$ is Lorentzian. Then there exists a constant $C>0$ depending only on $\alpha$, $\Sigma$ and the estimates in \autoref{lem:presspacelike}, \autoref{lem:HSbound} and \autoref{lem:tildeubound} such that for all $t\in[0,T)$, we have $S<C\|x\|$. 
\end{prop}
\begin{proof}
We first estimate the maximum of $S$ a the boundary.
Suppose that $S^2\|x\|^{-2}>2(1+\al^2)$. By \eqref{eq:hathnunu} we have
\begin{flalign*}
\widehat{h}\left(\frac{\nu^\Sigma}{\|\nu^\Sigma\|},\frac{\nu^\Sigma}{\|\nu^\Sigma\|}\right) &= q(x)\|x\|^{-1}\left(\frac{S^2\|x\|^{-2}}{\al^2+1}-1\right)>c_\Si\|x\|^{-1}\left(\frac{S^2\|x\|^{-2}}{\al^2+1}-1\right).
\end{flalign*}
Therefore, at a boundary maximum, as a result of \autoref{prop:timelikeboundaryspacelike}, \autoref{lem:HSbound}, \autoref{lem:tildeubound}, and equation \eqref{eq:xtotime}
\begin{flalign*}
\nabla_{\mu^\top}S &\leq\! S\left(\frac{-c_\Si}{\|x\|}\left(\frac{S^2\|x\|^{-2}}{\al^2+1}-1\right)+\frac{S|\alpha|}{2C_{\frac{S}{H}}+2t}\right)\leq \!\frac{S}{\|x\|}\left(-c_\Sigma\left(\frac{S^2\|x\|^{-2}}{\al^2+1}-1\right)+C_1 S\|x\|^{-1}\right)
\end{flalign*}
for some $C_1=C_1\left(\alpha, c_{\frac{S}{H}}, C_{\frac{\|x\|^2}{1+t}}\right)$. Therefore $\n_{\mu^\top} S <0$ if $S\|x\|^{-1}>C_2$ for some $C_2=C_2(c_\Si,C_1)>\sqrt{2(1+\al^2)} $.  Using \eqref{eq:xtotime}, this implies that there exists a $C_3 = C_3(C_2,C_{\frac{\|x\|^2}{1+t}})$ such that any global maxima of $S^2$ at the boundary satisfy $S^2<C_3(1 +t)$.

We now deal with the interior maximum. Namely we have
\[\ho S \leq -|A|^2 S +2H\leq -\frac 1 2 H^2 S +2H\leq 2S^{-1},\]
so
\[\ho S^2 \leq 4 -\frac{|\n S^2|^2}{2S^2}.\]
As a result,  there exists a $C_4=C_4(M_0, C_3)$ such that on $\overline M_t$ we have $S^2\leq C_4+\max\{4, C_3\}t$ by the maximum principle. Therefore the claim follows by \autoref{lem:tildeubound} and \eqref{eq:xuLipschitzequiv}.
\end{proof}

\begin{cor}\label{cor:unifspacelike}
Suppose that the assumptions of \autoref{thm:maintheorem} hold and that $\Sigma$ is Lorentzian. Then there exists a $C_v>0$ depending only on $\alpha$, $\Sigma$ and the estimates in \autoref{lem:presspacelike}, \autoref{lem:HSbound} and \autoref{lem:tildeubound} such that for all $t\in[0,T)$, $v<C_v$.
\end{cor}
\begin{proof}
    For all $x\in\Sigma$, $-\ip{e_{3}}{x\|x\|^{-1}}\leq C(\Si)$, so by estimating literally as in the proof of \autoref{lem:presspacelike} but replacing $\mu$ with $x\|x\|^{-1}$ and $\alpha$ with the estimate $C(\Si)$ we see that $v<\wh{C}(\Si)S\|x\|^{-1}$. The estimate now follows from \autoref{prop:uniformspacelikeness}.
\end{proof}
\section{Proof of \autoref{thm:maintheorem}}
\label{sec:ProofOfTheorem}
We have the following improvement on \autoref{lem:HSbound}.
\begin{lemma} \label{lem:HSimprovement}
On the rescaled flow $\wt{M}_\tau$ we have
\[\left|\frac{\wt{H}}{\wt{S}}-\frac 1 2\right|\leq C e^{-\tau}.\]
\end{lemma}
\begin{proof}
By scaling properties we have $\frac{\wt{H}}{\wt{S}} =(1+t)\frac{H}{S}$. Using the above evolutions gives
\[\ho (1+t)\frac{H}{S}=-2\frac{H}{S}\left((1+t)\frac H S-\frac{1}{2}\right)+2\ip{\frac{\n S}{S} }{\n \frac{H}{S}}.\]
Writing $\phi= (1+t)\frac{H}{S}-\frac 1 2$,
\[\ho \phi=-2\frac{H}{S}\phi+2\ip{\frac{\n S}{S} }{\n \phi}.\]
Using \autoref{lem:HSbound}, 
\begin{flalign*}
\ho \phi^2&=-4\frac{H}{S}\phi^2+\ip{4\phi\frac{\n S}{S}-2\n\phi }{\n \phi}\leq -\frac{2}{C_{\frac{S}{H}}+t}\phi^2+2\ip{\frac{\n S}{S} }{\n \phi^2}, 
\end{flalign*}
and therefore
\begin{flalign*}\ho \phi^2(C_{\frac{S}{H}}+t)^2&\leq 2(C_{\frac{S}{H}}+t)\phi^2-2(C_{\frac{S}{H}}+t)\phi^2+2\ip{\frac{\n S}{S} }{\n (C_{\frac{S}{H}}+t)^2\phi^2}\\
&=2\ip{\frac{\n S}{S} }{\n (C_{\frac{S}{H}}+t)^2\phi^2}.
\end{flalign*}
Applying the maximum principle leads to
\[\left|(1+t)\frac{H}{S}-\frac 1 2\right|\leq \frac{C(M_0)}{C_{\frac{S}{H}}+t}.\]
\end{proof}

We define the $k^\text{th}$ compatibility condition by
\[\frac{d^k}{dt^k}{b}(x,D\wt{\rho})\Big|_{t=0}=0,\]
where ${b}$ is as in \eqref{eq:MCFrescaledrho}.

We now prove the following which implies \autoref{thm:maintheorem}.
\begin{thm}\label{thm:maintheoremPDE}
Suppose that $\Sigma$ is the non-degenerate boundary of a strictly convex cone $\cC \subset \mathbb{R}^3_1$ and that $M_0\subset \ov{\mathcal{C}} $ is strictly mean convex,  spacelike and graphical with graph $\wt{\rho}_0$. Suppose furthermore that $\wt{\rho}_0$ satisfies the boundary condition and all compatibility conditions up to $k^\text{th}$ order. Then: 
\begin{enumerate}
\item A unique bounded solution 
\[\wt{\rho}\in C^{2k+\alpha;\frac{2k+\alpha}{2}}(\Omega\times [0,\infty)))\cap C^\infty(\Omega \times (0,\infty))\] of \eqref{eq:MCFrescaledrho} exists for all rescaled time $\tau$. The solution is bounded and given any fixed $\ep>0$, we have uniform $C^{l;\frac{l}2}(\Omega\times [\ep,\infty))$ for all $0\leq l\in \mathbb{N}$.
\item  The solution $\wt{\rho}$ converges uniformly and smoothly to a piece of an expanding solution to MCF satisfying the boundary conditions as $t \ra \infty$.
\end{enumerate}
\end{thm}
\begin{proof}
The first part of this  follows from uniform parabolicity and obliqueness estimates and an application of standard theory, as in \cite[Section 2.6 and 6.1]{Schnuerer2002}. We now give more details. 

One obtains short time existence as in \cite[Theorem 2.5.7, page 106]{Gerhardt:/2006} where we also linearise the boundary conditions in an identical way and use \cite[Theorem 5.3, page 320]{LadyzenskayaSolonnikovUraltseva:/1968} to ensure suitable existence and estimates on the linearised problem for $\tau\in[0,\wt{T}_1)$ and some small $\wt{T}_1>0$. We now suppose that $\wt{T}<\infty$ is the maximal time interval on which a graphical solution exists.

In case $\si=-1$, the requirements for uniform parabolicity and obliqueness of \eqref{eq:MCFrescaledrho} in \autoref{prop:parabolicityobliqueness} have been shown to be satisfied in  \autoref{lem:presspacelike}, \autoref{prop:UniformGraphSpacelikeCone} and \autoref{lem:tildeubound} respectively.

For the case $\si=1$, the first condition in \autoref{prop:parabolicityobliqueness} follows from  \autoref{cor:unifspacelike}. Properties of timelike vectors imply that there is a constant $C=C(\Sigma)>0$ such that $Su^{-1}\geq-C\ip{\nu}{x\|x\|^{-1}}\geq C>0$ and so the second condition in \autoref{prop:parabolicityobliqueness} is satisfied. The final condition in \autoref{prop:parabolicityobliqueness} is satisfied due to  \autoref{lem:tildeubound}, and so uniform parabolicity and obliqueness follows.

Furthermore, using these estimates we also have
\[c_{\wt{\rho}}\leq \wt{\rho}\leq C_{\wt{\rho}} \qquad\text{ and }\qquad |D\wt{\rho}|^2<\frac{v^2}{S^2u^{-2}}<C_v^2c_{Su^{-1}}^{-2}.\]
By Nash--Moser--De Giorgi PDE estimates \cite[Lemma 13.22, p353]{Lieberman:/1998} we now have uniform estimates in $C^{1+\be;\frac{1+\be}{2}}$, and, as in \cite[Section 2.6 and 6.1]{Schnuerer2002}, we may bootstrap further to obtain uniform $C^{k+\be;\frac{k+\be}{2}}(\Omega\times[0, \wt{T}))$ (for $k$ as in the statement of the theorem). Furthermore, for any $\ep>0$, we also have estimates on $C^{l+\be;\frac{l+\be}{2}}(\Omega\times[\ep, \wt{T}))$ depending only on $\ep$ and $l$. As a result we may smoothly extend our solution to the interval $[\ep,\wt{T}]$ with $u(\wt{T})$ satisfying compatibility conditions to all orders. Therefore we may apply short time existence again to contradict the maximality of $\wt{T}$. Therefore $\wt{T}=\infty$ and we have the claimed smooth estimates.

By  \autoref{lem:HSimprovement} we have exponential decay of $\wt{\rho}_\tau = \frac{\wt{H}}{\wt{S}}-\frac 1 2$ and so there exists a subsequence of times such that $\wt{\rho}$ converges to a solution of the stationary equation $\wt{H} = \frac{\wt{S}}{2}$ which we will write $\wt{\rho}_\infty$. However, by integration we see that for $\tau_1,\tau_2>0$, \[|\wt{\rho}(x,\tau_1)-\wt{\rho}(x,\tau_2)|<Ce^{-\min\{\tau_1,\tau_2\}},\]
 and taking the same subsequential limit in $\tau_2$, 
 \[|\wt{\rho}(x,\tau)-\wt{\rho}_\infty|<Ce^{-\tau},\]
 implying full uniform convergence of the rescaled equation. Smooth convergence now follows from our smooth estimates and interpolation.
\end{proof}

\providecommand{\bysame}{\leavevmode\hbox to3em{\hrulefill}\thinspace}
\providecommand{\MR}{\relax\ifhmode\unskip\space\fi MR }
\providecommand{\MRhref}[2]{%
  \href{http://www.ams.org/mathscinet-getitem?mr=#1}{#2}
}
\providecommand{\href}[2]{#2}

\end{document}